\documentclass[a4paper, 12 pt]{amsart}
\usepackage{geometry}
\usepackage{marginnote}
\usepackage[utf8]{inputenc}
\usepackage[english]{babel}
\usepackage{amsthm}
\usepackage{amssymb}
\usepackage{enumerate}
\usepackage{graphicx}
\usepackage{float}
\usepackage{bbm}
\usepackage{amsmath}
\usepackage{comment}
\usepackage{bm}
\usepackage{extarrows}
\usepackage{subfigure}
\usepackage{mathtools}
\usepackage{color}
\usepackage{amsbsy, mathrsfs, amsfonts, mathdots}
\usepackage{hyperref}

\hypersetup{
    pdftitle={Memory Estimation of Inverse Operators},    
    pdfauthor={Anatoly Baskakov}{Ilya Krishtal},     
    pdfsubject={Memory Estimation of Inverse Operators},   
    pdfcreator={Ilya Krishtal},   
    pdfproducer={Ilya Krishtal}, 
    pdfkeywords={Banach modules}{Beurling spectrum}{Wiener's lemma}, 
     colorlinks,%
    citecolor=blue,%
    filecolor=black,%
    linkcolor=magenta,%
    urlcolor=black
}

\newtheorem{theorem}{Theorem}[section]
\newtheorem{corollary}[theorem]{Corollary}
\newtheorem{proposition}[theorem]{Proposition}
\newtheorem{lemma}[theorem]{Lemma}

\theoremstyle{definition}
\newtheorem{definition}[theorem]{Definition}
\newtheorem{example}[theorem]{Example}
\newtheorem{examples}[theorem]{Examples}

\theoremstyle{remark}
\newtheorem{remark}[theorem]{Remark}

\numberwithin{equation}{section}

\newcommand{\R}{\mathbb{R}}
\newcommand{\Z}{\mathbb{Z}}
\newcommand{\N}{\mathbb{N}}
\newcommand{\CC}{\mathbb{C}}
\newcommand{\HH}{\mathcal {H}}

\def\ra{\rangle}
\def\la{\langle}
\def\HH{\mathcal H}
\def\P{\mathscr P}
\def\I{I}
\def \F {\mathcal{F}}
\renewcommand{\I}{\mathscr {I}}

\title{Phaseless reconstruction from space-time samples}
\author{A.~Aldroubi, I.~Krishtal, S.~Tang}


\begin{document}
\maketitle

\date{\today}

\begin{abstract}
 Phaseless reconstruction from space-time samples is a nonlinear problem of recovering a function $x$ in a Hilbert space $\HH$ from the modulus of linear measurements $\{\lvert \la x, \phi_i\ra  \rvert$, $ \ldots$, $\lvert \la A^{L_i}x, \phi_i \ra \rvert : i \in\mathscr I\}$, where $\{\phi_i; i \in\mathscr I\}\subset \HH$ is a set of functionals on $\HH$, and $A$ is a bounded operator on $\HH$ that acts as an evolution operator.  In this paper, we provide various sufficient or necessary conditions for solving this problem, which has connections to $X$-ray crystallography, the scattering transform, and deep learning. 

\end{abstract} 

\section{Introduction}

\subsection{The phaseless reconstruction problem}
 To perform phaseless reconstruction, one needs to find an unknown signal $x \in \HH$ from the modulus of the linear measurements 
\begin {equation} \label {pahselessRec}
\{\vert \la x,f\ra \vert, \; f \in \F \subset \HH\}, 
\end {equation} where $\HH$ is a real or complex separable Hilbert space, and $\F$ is a well-chosen or known countable subset in $\HH$. Since for any $c$ such that $\vert c\vert=1$ the signals $x$ and $cx$ have the same measurements $\{\vert \la x,f\ra \vert, \; f \in \F \subset \HH\}$,  one can only hope to reconstruct $x$ up to some unimodular constant $c\in \mathbb{K}\in \{\mathbb{R, C}\}$, $\lvert c\rvert=1$, which is typically called a \emph{global phase}. Therefore, to get a well-posed problem, one considers the equivalence relation on $\HH$ defined by $x \sim y $ if and only if  $y=cx$ for some $c\in \mathbb{K}$, $\lvert c\rvert=1$. In particular, when $\mathbb{K}=\mathbb{R}$ and $\HH=\mathbb{R}^n$, $y\sim x$ if and only if $y=x$ or $y=-x$. 
Clearly,
the necessary condition for  phaseless reconstruction is  then
the injectivity of
the mapping $T_{\mathcal{F}}: (\HH/\sim) \rightarrow \ell^2(\I)$, 
$$\widetilde x \rightarrow (\lvert \langle x, f\rangle \rvert)_{f\in \F}, $$ where $\mathscr I\subset \N$ is an indexing set with $\vert \mathscr I \rvert=\vert \F \rvert$, and $(\HH/\!\sim)$ is the quotient space. Hence, the first milestone in solving the phaseless reconstruction problem is finding conditions on $\F$ such that the map $T_{\mathcal{F}}$  is injective,
in which case we say that the set $\mathcal{F}$ \emph{does phaseless reconstruction} on $\HH$. 
This aspect of the problem was studied, for example, in \cite {BCE06, BCMA14,  BP17, CCQH16, CDMV15,PJ14, MV15,  PYB13,  Tha11}. The other important aspect of phaseless reconstruction is finding numerical algorithms for the recovery of the signal from phaseless measurements. Various approaches to this part of the problem can be found in  
\cite{BAMD14, BBCE09, BE16, CESV15,CSV13, DH14, EM14, IVW16, IVW17}.

The phaseless reconstruction problem and the equivalent  {\it  phase retrieval problem} appear in many different applications, for example, in imaging science \cite{BCM14, FL12, FL13}. 
The most well-known application 
 is in X-ray crystallography \cite {RD86, F78, F82}, where the signal $x\in L^2(\Lambda)$ is the density  of electrons on a crystal lattice $\Lambda \subset \R^3$ and the measurements $\{\vert \la x,f\ra \vert, \; f \in \F \subset \HH\}$ are proportional to the modulus of the Fourier coefficients $\{\vert \hat x(\omega) \vert: \omega \in \Omega \}$, where $\Omega$ is the reciprocal lattice of $\Lambda$. More recently, the phaseless reconstruction problem is finding applications in the artificial intelligence area of  deep learning and convolution neural networks (see, e.g., \cite {MW15, WB15} and the references therein). 


\subsection{The phaseless reconstruction problem in Dynamical sampling}
The object of study
in {\it Dynamical sampling} is an unknown vector $x \in \HH$ evolving  by iteration under the action of a bounded operator $A$ on $\HH$. In other words, at time $n$ the signal $x$ evolves to become  $x_n=A^n x$.  For 
$\Phi=\{\phi_i: i \in \mathscr I\}\subset \HH$,  
  the first problem of dynamical sampling is  to find conditions on $A$, $\Phi,$ and $L_i$ that ensure that any $x\in\HH$ can be recovered  from the measurements   
\begin{equation}\label{DynsamplePhi}
Y=\{\la x, \phi_i\ra  , \ldots, \la A^{L_i}x, \phi_i \ra : i \in\mathscr I\}.
\end{equation}
This problem was studied in \cite{AT14, ACMT14, ACCMP16, ADK13,  D14,  Lv09, Phi17, RCLV11, T16,T17}. 

The phaseless reconstruction problem in dynamical sampling combines the two problems we have discussed. More precisely, one seeks to find conditions on $A$, $\Phi,$ and $L_i$ such that any $x\in \HH$ can be recovered up to a global phase from the measurements  
 \begin{equation}\label{samplePhi}
Y=\{\lvert \la x, \phi_i\ra  \rvert, \ldots, \lvert \la A^{L_i}x, \phi_i \ra \rvert : i \in\mathscr I\}.
\end{equation}

If $\HH=\R^d$, $\mathscr I \subset \{0,1,2,\ldots, d-1\}$, and $\phi_i=e_i$ are the standard basic vectors in $\HH$, then the problem reduces to finding $x$ up to a sign from the phaseless space-time samples 
\begin{equation}\label{sample}
Y=\{\lvert x(i) \rvert, \ldots, \lvert A^{L_i}x(i) \rvert : i\in \mathscr I\}.
\end{equation}

 It is not difficult to see that  
 the signal recovery from \eqref {samplePhi} can be reformulated using the set
 \begin{equation}\label{vectors}
\mathcal F=\{\phi_i, A^*\phi_i,\ldots,(A^*)^{L_i}\phi_i: i\in\mathscr I\}
\end{equation}
 in $\HH$ in the following way.

\begin{lemma}\label {PhRec} Any $x \in \HH$ can be  determined  uniquely up to a global phase
from the  measurements 
 \eqref{samplePhi} if and only if  the set of vectors 
\eqref{vectors} does phaseless reconstruction in $\HH$. 
\end{lemma} 
Of course, stability of the reconstruction is  important in applications. For a finite dimensional space, if $\F$ in \eqref {vectors} does phaseless reconstruction, then $\F$ is always a frame, and stability of reconstruction is always feasible. It is well known, however, that when $\HH$ is infinite dimensional, stability is very delicate and cannot be achieved in general \cite{CCD16}.

\subsection {Contributions and Organization}

\subsubsection {Contributions}  In this paper, we consider the problem of phaseless reconstruction in dynamical sampling in real Hilbert spaces. When the Hilbert space $\HH$ is  finite dimensional, we find conditions on the evolution operator $A$, the sampling functionals $\phi_i$, and the number of time levels $L_i$ for each $i\in\mathscr I$ such that  phaseless reconstruction from \eqref {samplePhi} is feasible. The conditions on $A$ are given in terms of the spectrum $\sigma(A)$ of the operator $A$, whereas the conditions on the functionals $\phi_i$ and the time levels $L_i$ depend on the manner in which the supports of the functionals interact with $\sigma(A)$. 
The  diagonalizable case is presented in Theorem \ref {GenRealMat} and Remark \ref {positiveEV}. The general case for finite dimensional $\HH$ is presented in Theorem \ref {GenRealMatJorD}. An application to the special case of a convolution operator is given in Corollary \ref{posconv}. 

For the infinite dimensional case, we only consider  operators $A$ that are diagonalizable and reductive. By diagonalizable,  we mean the operators that are unitarily equivalent to a diagonal matrix  on $\ell^2(\N)$. Reductive operators are those that have the property that if $V$ is an invariant subspace for $A$, then $V$ is also invariant for $A^*$; self-adjoint operators, for example, are reductive. For the case of diagonalizable reductive operators, the conditions for phaseless reconstruction are similar to the ones in the finite dimensional case and are presented in Theorems \ref {GenRealMatInf} and \ref {GenRealMatInfPos}. 

\subsubsection {Organization}  In Section \ref {NP}, we present  the concepts that are crucial for understanding  our main results, such as Theorem \ref {spanning}.
In Subsection \ref {IRM}, we introduce the classes of iteration regular  matrices and totally full spark matrices. We also state and prove Proposition \ref {iterregD} that relates these two notions. 
Another fundamental concept is that of local complementarity in Definition \ref {LC}. Subsection \ref {JMAP} is devoted to some notation related to Jordan forms and associated projections.

The main results of the paper are presented in Section \ref {Main}. Their proofs are relegated to Section \ref{mainproof}. They use  Theorems \ref {tadcul} and \ref {complete} in \cite {ACMT14} that give necessary and sufficient conditions on $A$, $\phi_i$, and $L_i$ for the exact reconstruction from measurements  \eqref {DynsamplePhi}. Outcomes of a numerical experiment on synthetic data are presented in Section \ref {NE}.

\section {Notation and Preliminaries} 
\label {NP}

\subsection{The complement property}
A result in  \cite{BCE06} gives a useful criterion for  a set $\mathcal{F}\subset \HH$ to do phaseless reconstruction whenever $\HH$ is a real Hilbert space. Specifically, the following characterization was obtained there (we include the proof since the result is a cornerstone of the theory we develop here).  

\begin{theorem}[\cite{BCE06}]\label{spanning} A subset $\mathcal{F}$ of a real Hilbert space $\HH$ does phaseless reconstruction on $\HH$ if and only if for any disjoint partition of $\mathcal{F}=\mathcal{F}_1 \cup \mathcal{F}_2$,  $\overline {span \mathcal{F}_1}=\HH$  or $\overline {span \mathcal{F}_2}=\HH$.
\end{theorem}

\begin{proof}
Assume that there is a disjoint partition $\mathcal{F}=\mathcal{F}_1 \cup \mathcal{F}_2$ such that neither $\overline {span \mathcal{F}_1}=\HH$  nor $\overline {span \mathcal{F}_2}=\HH$. Let $x_i \in  \HH\setminus\{0\}$  be such that $x_i \perp \mathcal F_i$, $i =1,2$. Then $x_1+x_2 \neq \pm(x_1-x_2)$, and  $\lvert \la x_1+x_2,f \ra \rvert =\lvert \la x_1-x_2,f \ra \rvert$ for all $f \in \F$. Thus $x_1+x_2$ cannot be distinguished from $x_1-x_2$ by phaseless measurements generated by $\mathcal F$.

Conversely, assume that, for any disjoint partition $\mathcal{F}_1, \mathcal{F}_2$ of $\mathcal{F}$,  $\overline {span \mathcal{F}_1}=\HH$  or $\overline {span \mathcal{F}_2}=\HH$, and let $x, y\in \HH$ be two vectors that produce the same phaseless measurements, i.e., $|\langle x, f\rangle| = |\langle y, f\rangle|$ for all $f \in \mathcal F$. Consider the set $\mathcal F_1=\{f \in \mathcal F: \langle x, f\rangle  \ne 0,  \, \langle x, f\rangle = \langle y, f\rangle\}$, and $\mathcal F_2=\mathcal F\setminus \mathcal F_1=\{f \in \mathcal F: \langle x, f\rangle = -\langle y, f\rangle\}$.  Then, if  $\overline {span \mathcal{F}_1}=\HH$, we have $x=y$. If   $\overline {span \mathcal{F}_2}=\HH$, then $x=-y$.
\end{proof}



When $\HH=\R^n$, we also have the following corollary.
\begin{corollary}[\cite{BCE06}] \label {loccompreal} 
Assume that a set $\mathcal{F}=\{f_1,f_2, \ldots, f_m\}$ does phaseless reconstruction on $\mathbb{R}^n$. Then $m \geq 2n-1$. 
\end{corollary}

The theorem above suggests the following definition.

\begin{definition}
\label {CompProp}
Let $\mathcal{F}\subset \HH$ be a system of vectors in $\HH$. We say that $\mathcal{F}$ satisfies the \emph{complement property in $\HH$} if  any disjoint partition  $\mathcal{F}=\mathcal{F}_1 \cup \mathcal{F}_2$ satisfies $\overline {span \mathcal{F}_1}=\HH$  or $\overline {span \mathcal{F}_2}=\HH$. 
\end{definition}

\begin{remark}\label{expnec}
 According to Theorem \ref {spanning}, a set $\mathcal{F}$ does phaseless reconstruction on a real Hilbert space $\HH$ if and only if $\mathcal F$ has the complement property in $\HH$. The equivalence is no longer true, however, if we wish to recover any $x$ in  a complex Hilbert space $\HH$. In this case,  the complement property is necessary but not sufficient. For example, consider  vectors $\psi_1 = \left(
\begin{array}{c} 1 \\  0\end{array}\right)$ , $\psi_2 = \left(
\begin{array}{c} 0 \\  1 \end{array}\right)$,  and $\psi_3 = \left(
\begin{array}{c} 1 \\  1 \end{array}\right)$ in $\CC^2$. Then $\{\psi_1,\psi_2, \psi_3\}$ has complement property in $\CC^2$, but it does not allow phaseless reconstruction since $\left(
\begin{array}{c} 1 \\  i \end{array}\right)$ and $\left(
\begin{array}{c} 1 \\  -i \end{array}\right)$ have the same measurements. 
\end{remark}

The following definition will lead us to a necessary condition for phaseless reconstruction.

\begin{definition}\label {LC} 
Let $\mathscr P = \{P_j: j\in \Gamma\}$ be a family of projections in a Hilbert space $\HH$ and $\Psi = \{\psi_i: i\in\mathscr I\}$ be a set of vectors in $\HH$. We say that $\Psi$ is \emph{locally complementary} with respect to $\mathscr P$ if for every partition $\{\mathscr I_1, \mathscr I_2\}$ of $\mathscr I$ there exists $\ell\in\{1,2\}$ such that 
for each $j \in \Gamma$ the set
\begin{equation}\label{projectvectors1}
\{P_j\psi_i : i \in \mathscr I_\ell\} 
\end{equation}
spans the range $E_j$ of $P_j$.
\end{definition}

The following proposition follows immediately from the definition.

\begin {proposition}\label{locness}
Let $\HH$ be a complex Hilbert space,    $\Psi = \{\psi_i: i\in\mathscr I\}$ be a set of vectors in $\HH$, and 
$\mathscr P = \{P_j: j\in \Gamma\}$ be a family of  projections. If $\Psi$ has a complement property in $\HH$ then it is locally complementary with respect to $\mathscr P$.
\end {proposition}

\subsection{Iteration regular matrices} \label {IRM}

In this section, we introduce a new class of matrices that is especially amenable for phaseless reconstruction in dynamical sampling.  
We begin with the following definitions.

\begin {definition} \label {minpol} Consider a bounded operator $A\in B(\HH)$ and a vector $x\in\HH$.
\begin {enumerate}
\item If there exists a non-zero polynomial $q$ such that $q(A)=0$, then   the  monic polynomial $p$ of the smallest degree such that $p(A) = 0$ will be called the {\em minimal (annihilating) polynomial of $A$} and will be denoted by $p^A$. Whenever $p^A$ exists, $r^A$ will denote its degree; otherwise we let $r^A=\infty$. 

\item  If there exists a non-zero polynomial $q$ such that $q(A)x=0$, then  the  monic polynomial $p$ of the smallest degree such that $p(A)x = 0$ will be called the {\em (A,x)-annihilator} and will be denoted by $p^A_x$. Whenever $p^A_x$ exists, $r^A_x$ will denote its degree; otherwise we let $r^A_x=\infty$.  


\end{enumerate}
\end {definition}

Observe that in the above definition we always have $r^A_x \le r^A$, $r^0 = 1$, and $r^A_0 = 0$.

\begin {definition} \label {PartAnn} Assume that a bounded operator $A\in B(\HH)$ has a minimal polynomial $p^A$.  A non-zero  polynomial $p$ is a \emph{$k$-partial annihilator} of $A$, $k\in\mathbb N$, if it has at least $k$ roots in common with $p^A$ counting the multiplicity or, in other words, if $p$ and $p^A$ have a common divisor of degree $k$.
\end {definition}

\begin {definition} \label {IterReg} Let $A\in\mathbb K^{n\times n}$ be a  matrix.
The matrix $A$ is called \emph{iteration regular} if for all  {$k\in \mathbb N$}, any $k$-partial annihilator of $A$  of degree at most {$r=\max \{1,2k-2\}$ }has at least $k+1$ non-zero coefficients.  
\end {definition}

{\begin {examples}  \label {expiterreg}${}$ We illustrate the above definition by the following examples.
\begin {enumerate}
\item Consider the identity  matrix $I = I_n$ in $\mathbb K^n$. Then $p^I(x)= x-1$ and all partial annihilators are $1$-partial annihilators of the form  $c(x-1)$ with $c\in\mathbb R\setminus\{ 0\}$.  These polynomials have $2$ non-zero coefficients, and this immediately implies that the matrix $I_n$ is iteration regular. 
\item Assume that $A$ is a singular $n\times n$ matrix on $\mathbb K$. Then its minimal polynomial is divisible by $x$. Hence, $q(x) =x$ is  a $1$-partial annihilator of degree $1$ that has only one non-zero coefficient. Therefore,  $A$ is  not iteration regular. Observe that for invertible matrices we only need to check $k$-partial annihilators for cases when $k > 1$.
\item Consider the $2\times 2$ diagonal matrix $D=diag (-1,2)$. Then $p^D(x)=(x+1)(x-2)$. Since $D$ is invertible, we do not need to check the $1$-partial annihilators. Therefore, it suffices to check only the $2$-partial annihilators of degree $2$. These  are given by $c(x+1)(x-2)$,  $c\in\mathbb R\setminus\{ 0\}$. Clearly, they have  $3$ non-zero coefficients and $D$ is iteration regular. 
\item Assume that $A$ is any $n\times n$ matrix on $\mathbb K$, such that its minimal polynomial $p^A$ has a pair of  purely imaginary conjugate roots. Then $A$ is not iteration regular because it has a $2$-partial annihilator with only two non-zero coefficients.
\end {enumerate}
\end {examples}
}



Clearly, the property of iteration regularity is a purely spectral property; it is invariant under similarity of matrices. Therefore, to determine if a given matrix is iteration regular, it suffices to consider  the Jordan canonical form
of the matrix. In general, however, it may still be very difficult.  We will provide some sufficient conditions that are of interest.

We begin with  a sufficient condition for  iteration regularity of a diagonal matrix. We will need the following definition. 

{ \begin {definition}
\label {spark}
Let  $B \in \mathbb{K}^{m\times n}$, be such that $m \leq n$.
\begin {enumerate}
\item  The \emph{spark} of $B$ is the size of the smallest  linearly dependent subset of its columns, i.e.,
$$Spark(B)=min\{\lvert\lvert f \rvert\rvert_0:  f \in \mathbb{K}^n, Bf=0, f \neq 0\}.$$
\item The matrix $B$ has \emph{full spark} if $spark(B) = m+1$.
\item  By the spark of a set $\{\phi_1, \phi_2, \ldots,
\phi_n\}\subset \mathbb{K}^m$ we mean the spark of the matrix formed by these vectors as columns.
\item A matrix $B \in \mathbb{K}^{m\times n}$, $m, n\in \N^+$, has \emph{totally full spark} if every one of its square submatrices is invertible.
\end {enumerate}
\end {definition}
}
\begin{proposition}
\label {iterregD}
Let $D\in\mathbb K^{n\times n}$ be a diagonal matrix with $d$ distinct eigenvalues
$\lambda_1$, \ldots, $\lambda_d$. Assume that the $d\times (2d-1)$ matrix
$\Lambda= (\lambda_j^\ell)$, 
 $j = 1,\ldots, d$, $\ell = 0,\ldots, 2d-2$, has totally full spark. Then $D$ is iteration regular.
 \end{proposition}
 
 \begin{proof}
 Since the matrix $\Lambda$ has totally full spark, the matrix $D$ is invertible and we only need to check $k$-partial annihilators for $k\ge 2$.
 Let $p$ be a $k$-partial annihilator of $D$ of degree at most $2k-2$, $k\ge 2$. Then its coefficients are in the kernel of the submatrix of the matrix $\Lambda$ that is comprised of the columns that are generated by the eigenvalues that are roots of $p$. Since $\Lambda$ has totally full spark,  we cannot  have fewer than $k+1$ non-zero coefficients in $p$. 
 \end{proof}

Next, we consider the case of a Jordan cell. We will use the notation
\begin{equation}\label{Nil2}
N_{s}=\left( \begin{array}{cc}
0 & 0 \\
I_{s-1}& 0
\end{array} 
\right) = \left( \begin{array}{cccccc}
0 & 0 &\cdots & 0 &0&0\\
1& 0&\cdots& 0&0&0\\
0& 1&\cdots& 0&0&0\\
\vdots & \vdots & \ddots&\vdots&\vdots&\vdots\\
0&0&\cdots&1&0&0 \\
0&0&\cdots&0&1&0
\end{array} 
\right)
\end{equation}
for the $s\times s$  nilpotent matrix $N_s$, $s\in\N$. By $I_s$ we shall denote the $s\times s$  identity matrix.

\begin{proposition}
Let $J = J(\lambda)\in\mathbb K^{n\times n}$ be the Jordan cell given by $J = \lambda I_n+N_n$. Then $J(\lambda)$ is iteration regular
if and only if $\lambda\neq  0$.
 \end{proposition}

 \begin{proof}
 We only need to prove that if $\lambda\neq 0$ then $J$ is iteration regular. The other implications stated in the theorem immediately follow since iteration regularity implies invertibility.
 
 Assume $J$ is invertible, i.e., $\lambda\neq 0$.
 Let $e_1$ be the first standard basic vector in $\mathbb K^n$. Consider the $n\times (2n-1)$ matrix $\Lambda = \Lambda(\lambda)$ whose  first column $f_1$ equals $e_1$ and each $j$-th column $f_j$ satisfies $f_j = J(\lambda)f_{j-1}$, $j\ge 2$. We shall denote by $\Lambda_k$ the submatrix of $\Lambda$ formed by the first $k$ rows and $2k-1$ columns: 
\[
\Lambda_{k}=\left( \begin{array}{ccccccc}
1 & \lambda &\cdots & \lambda^{k-1} &\lambda^k&\cdots&\lambda^{2k-2}\\
0& 1&\cdots& (k-1)\lambda^{k-2}&k\lambda^{k-1}&\cdots&(2k-2)\lambda^{2k-3}\\
\vdots & \vdots & \ddots&\vdots&\vdots&\ddots&\vdots\\
0&0&\cdots&1&k\lambda&\cdots&{{2k-2}\choose{k}}\lambda^{k-1} \\
\end{array} 
\right).
\] 
 Since $J$ is invertible, we do not need to check the $1$-partial annihilators.   Let $p$ be a $k$-partial annihilator of $J$ of degree at most $2k-2$, $k\ge 2$, and $c \in\mathbb K^n$ be the vector of its coefficients (amended with zeroes, if necessary). Observe that we must have $\Lambda_k c = 0$. Thus, the matrix $J$ is iteration regular if and only if each $k\times k$ submatrix of $\Lambda_k$ is invertible, $k=1,\ldots, n$, or, in other words, when $\Lambda_k$ has full spark. The latter, however, holds if and only if $\lambda \neq 0$. 
  
The above is a more or less standard fact from linear algebra, we sketch the idea of the proof for completeness.
For $\lambda > 0$, the matrix $\Lambda(\lambda)$ is upper  strictly totally positive according to the definition in \cite[p.~47]{P10}. This follows from \cite[Theorem 2.8]{P10} by a simple inductive argument. Thus, every $k\times k$ minor of the matrix $\Lambda_k$ is positive when $\lambda >  0$. For $\lambda < 0$, it suffices to notice that, by definition of determinant, each  minor of $\Lambda(\lambda)$ has the same absolute value as the corresponding minor of the matrix $\Lambda(-\lambda)$.
 \end{proof}

Iteration regular matrices are useful for us because of the following special property possessed by their Krylov subspaces.

\begin{definition}
A  \emph{Krylov subspace} of order $r$ generated by a matrix $A$ and a nonzero vector $x\in\mathbb K^n$ is
$\mathcal K_r(A,x) = {\rm span} \{x, Ax, \ldots, A^{r-1}x\}.$ The \emph{maximal Krylov subspace}  of the matrix $A$ and  the nonzero vector $x\in\mathbb K^n$ is $\mathcal K_{\infty}(A,x) = {\rm span} \{x, Ax, \ldots\}.$
\end{definition}

\begin{proposition}\label{krylovspan}
Assume that $A\in\mathbb K^{n\times n}$ is iteration regular and a non-zero vector $x\in\mathbb K^n$ is such that the maximal Krylov subspace $\mathcal K_\infty(A,x)$ has dimension $k$. Then any $k$  vectors from the system {$\{x ,\ldots, A^{r}x\}$, $r=\max \{1,2k-2\}$,} form a basis in $\mathcal K_k(A,x) = \mathcal K_\infty(A,x)$.
\end{proposition}

\begin{proof}
The dimension $k$ of the maximal Krylov subspace $\mathcal K_\infty(A,x)$ is equal to the degree $r^A_x\le r^A$ of the $(A,x)$-annihilator $p^A_x$  (see Definition \ref  {minpol}~(2)) (note that  the polynomial $p^A_x$ divides the minimal polynomial $p^A$ of $A$). { The case for which $k=1$ is trivial.  It suffice to consider the case of $k \geq 2$}. Let $E=\{A^{J_i}x: i=1,\dots,k\}$ be $k$ vectors from $\{x, Ax,\ldots, A^{2k-2}x\}$ and consider the vanishing linear combination $$\sum\limits_{i=1}^k c_iA^{J_i}x=0.$$  The left hand side of the last identity is a polynomial $q(A)$ of degree no larger than $2k-2$ applied to $x$. Since $q(A)x=0$, the {  $(A,x)$-annihilator }   $p^A_x$ divides $q$. Therefore, $q$ has $k=r^A_x$ zeroes in common with $p^A$. Since $A$ is iteration regular, $r^A_x\le r^A$, and  $q$ has at most $k$ non-zero coefficients,  the polynomial $q$ and all its coefficients $c_i$ must be zero.
\end{proof}

\subsection {Jordan Matrices and Associated Projections} \label {JMAP} 
The proofs of our main results use the Jordan decomposition of a matrix as well as certain projections that are associated with them. 

Consider a Jordan matrix $J$
\begin{equation}\label{Jordan1}
J= \left( \begin{array}{cccc}
J_1 & 0 & 0 &0 \\
0 & J_2 & 0&0 \\
0 & 0 & \ddots&0 \\
0&0&0&J_d\end{array} \right)
\end{equation}

In \eqref{Jordan1}, for $s=1,\ldots,d$,  we have $J_s=\lambda_s I_s+M_s$ and $M_s$ is an $h_s\times h_s$ nilpotent block-matrix of the form:

\begin{equation}\label{Nil1}
M_s= \left( \begin{array}{cccc}
N_{s_1} & 0 & 0 &0 \\
0 & N_{s_2}& 0&0 \\
0 & 0 & \ddots&0 \\
0&0&0&N_{s_{r_s}}\end{array} \right)
\end{equation}
where each $N_{s_i}$ is a $t_{i}^{(s)}\times t_{i}^{(s)}$ cyclic nilpotent matrix of the form \eqref{Nil2}
with $t_{1}^{(s)} \geq t_{2}^{(s)}\geq \ldots \geq t_{r_s}^{(s)}$ and $t_{1}^{(s)} + t_{2}^{(s)}+\cdots+ t_{r_s}^{(s)}=h_s$. Also $ h_1 + \ldots + h_n = n$. The matrix $J$  has $n$ rows and distinct eigenvalues $\lambda_j , j = 1,\ldots, d.$

\begin {definition} \label {defJord} Let $k_j^{s}$ denote the index corresponding to the first row of the block $N_{s_j}$ from the matrix $J$,
and let $e_{{k}_{j}^{s}}$ be the corresponding elements of the standard basis of $\mathbb{C}^n$, so that each  $e_{{k}_{j}^{s}}$ is a cyclic vector associated to the respective block. We also define  $E_s=span\{e_{{k}_{j}^{s}}: j=1,\ldots, r_s\},$ for $s=1,\ldots,d$, and $P_s$ will denote the orthogonal projection onto
$E_s$. The family $\mathscr P_J=\{P_j: j=1\ldots d\}$ comprised of these projections will be called the \emph{penthouse family} of the matrix $J$. 
\end {definition}


\section{Main Results}
\label {Main}

\subsection {Finite dimensional case}
In the first theorem of this section, we present our most general sufficient conditions for phaseless reconstruction in dynamical sampling in the finite dimensional case.
We begin, however, with a proposition that gives a necessary condition. 

\begin {proposition} \label {NecConLC} Let $A\in \mathbb{R}^{n\times n}$ be such that $A^T=B^{-1}JB$, where $J\in\CC^{n\times n}$ is a Jordan matrix of the form \eqref{Jordan1}.  { Let also $\{\phi_i: i\in \mathscr I\}$ be a set of vectors in $\R^n$, $\Psi = \{\psi_i=B\phi_i: i\in \mathscr I\}\subset \CC^n$,} and $r_i =r^J_{\psi_i}$. If 
\begin{equation}\label{sete}
\mathcal E = \{\phi_i, (A^T)\phi_i,\ldots,(A^T)^{2r_i-2}\phi_i: i\in \mathscr I\}
\end{equation}
does phaseless reconstruction on $\mathbb{R}^n$ then the set $\Psi $  is locally complementary with respect to the penthouse family $\mathscr P_J$ {\color {green}(see Definition \ref{defJord})}. 
\end {proposition}

{If in addition to the  local complementarity condition in the previous proposition, we require that $A$ is iteration regular  (see Definition \ref {IterReg}), then we obtain sufficient conditions for phaseless reconstruction.}
\begin{theorem}\label{GenRealMatJorD}
Let $A\in \mathbb{R}^{n\times n}$ be such that $A^T=B^{-1}JB$, where $J\in\CC^{n\times n}$ is a Jordan matrix of the form \eqref{Jordan1}.  Let also { $\{\phi_i: i\in \mathscr I\}$ be a set of vectors in $\R^n$}, $\Psi = \{\psi_i=B\phi_i: i\in \mathscr I\}$ be a set of vectors in $\CC^n$ and $r_i =r^J_{\psi_i}$ {\color {green} (see Definition \ref {minpol})}.
 If $J$ is iteration regular {\color {green} (see Definition \ref {IterReg})} and  the set $\Psi $  is locally complementary with respect to the penthouse family $\mathscr P_J$ {\color {green}  (see Definition \ref{defJord})}  then the set of vectors 
{ \begin{equation}\label{vectors1}
\mathcal E = \{\phi_i, (A^T)\phi_i,\ldots,(A^T)^{2r_i-2}\phi_i: i\in \mathscr I\}
\end{equation}
}
does phaseless reconstruction on $\mathbb{R}^n$.  In other words, any $f\in \R^n$ can be uniquely determined up to a sign from its unsigned measurements $$Y=\{\lvert \la A^lf, \phi_i\ra \rvert:  i \in \mathscr I,\; l=0,\dots, 2r_i-2\}.$$
\end{theorem}

Using Proposition \ref {iterregD} and the above result, we get the following theorem for the case of diagonalizable matrices. 

\begin{theorem}\label{GenRealMat}
Let $A\in \mathbb{R}^{n\times n}$ be such that $A^T=B^{-1}DB$, where $D\in\CC^{n\times n}$ is a diagonal matrix with 
 $d$ distinct eigenvalues $\lambda_1$, \ldots, $\lambda_d$.
 Let also $\Lambda= (\lambda_j^\ell)$, 
 $j = 1,\ldots, d$, $\ell = 0,\ldots, 2d-2$, be a $d\times (2d-1)$ matrix comprised of the powers of these eigenvalues,
 $\Psi = \{\psi_i=B\phi_i: i\in\mathscr I\}$ be a set of vectors in $\CC^n$, and $r_i =r^D_{\psi_i}$. If the  matrix  $\Lambda$ has
 totally full spark and the set $\Psi $  is {locally complementary} with respect to the penthouse family $\mathscr P_D$     then the set of vectors 
\begin{equation}\label{vectors31}
\mathcal E = \{\phi_i, (A^T)\phi_i,\cdots,(A^T)^{2r_i-2}\phi_i: i\in \mathscr I\}
\end{equation}
does phaseless reconstruction on $\mathbb{R}^n$.  In other words, any $f\in \R^n$ can be uniquely determined up to a sign from its unsigned measurements {$$Y=\{\lvert \la A^lf, \phi_i\ra \rvert:  i \in \mathscr I,\; l=0,\dots, 2r_i-2\}.$$}
\end{theorem}

\begin {remark} Note that the matrix $A$ in Theorem \ref {GenRealMatJorD} or  \ref {GenRealMat} is necessarily invertible since an iteration regular matrix is always invertible { (see example \ref {expiterreg} (2))}. 
\end{remark}

\begin {remark} \label {necnotsuf} The condition on the matrix $\Lambda$  in Theorem \ref {GenRealMat} is not necessary. For example, let $D$ be a $2 \times 2$ diagonal matrix with the spectrum 
$\sigma(D) = \{1,-1\}$ so that $\Lambda = \left(
\begin{array}{rrrr} 1&1&1&1 \\  1&-1&1&-1\end{array}\right)$
does not have totally full spark. For this case,  ${\rm{rank}}\, P_j=1$, $j = 1,2$. Consider two vectors $\phi_1 = \left(
\begin{array}{c} 1 \\  1\end{array}\right)$ and $\phi_2 = \left(
\begin{array}{c} 1 \\  2\end{array}\right)$. Then
\[
\mathcal E = \left\{\left(
\begin{array}{c} 1 \\  1\end{array}\right), \left(
\begin{array}{c} 1 \\  -1\end{array}\right), \left(
\begin{array}{c} 1 \\  2\end{array}\right), \left(
\begin{array}{c} 1 \\  -2\end{array}\right)\right\}
\]
has complement property in $\R^2$. Hence,  $\mathcal E$ does phaseless reconstruction in $\R^2$.

However, local complementarity alone is not sufficient for the complementary property. For example, let $D$ be a $4\times 4$ matrix with the spectrum $\sigma(D) = \{\lambda_1,\lambda_2\}$, with $\lambda_1=1$ and $\lambda_2=-1$ each having the algebraic (and geometric) multiplicity $2$. Then ${\rm{rank}}\, P_j =2$, $j = 1,2$. Consider vectors
$\psi_i = \left(
\begin{array}{c} 1 \\  m \\ 1\\ m\end{array}\right)$, $m =1,2,3$. It is easy to check that the set of these three vectors is locally complementary with respect to $\{P_1, P_2\}$ but the set $\mathcal E$ { in \eqref {vectors31}} generated by them has only six vectors.  Therefore, by Corollary \ref {loccompreal},  $\mathcal E$ does not have the complementary property in $\R^4$ and hence does not do phaseless reconstruction.
Thus, local complementarity alone is not sufficient for phaseless reconstruction.
\end{remark}

\begin {remark} \label {positiveEV}
If all the eigenvalues $\lambda_j$, $j = 1,\ldots, d$ of the matrix $A$ in Theorem \ref {GenRealMat} are strictly positive, then the matrix  $\Lambda= (\lambda_j^\ell)$, 
 $j = 1,\ldots, d$, $\ell = 0,\ldots, 2d-2$, has totally full spark \cite{JP05}, and hence for this case local complementarity is sufficient {and necessary by Proposition \ref{NecConLC}} for the set \eqref {vectors31} to do phaseless reconstruction. 

\end{remark}

\begin{example}(Sampling at one node).
Assume $A\in \mathbb{R}^{n\times n}$ is similar to a diagonal matrix $D$ with $n$ distinct positive eigenvalues.  Clearly, we need at least one sampling sensor $\phi$ to do phaseless reconstruction.  Say $A^T=B^{-1}DB$, and let $\phi=e_i$. Then, as long as $Be_i$ is entrywise nonzero, it suffices to choose the sampling set $\mathscr I=\{i\}$ and take unsigned samples at time levels $t=0, \ldots 2n-1$,  to  recover $f$ up to a sign. 
\end{example}

An important case of an evolution process that is frequently encountered in practice is the so-called spatially invariant evolution process. In this case, the evolution operator $A$ is a circular convolution matrix  defined by a convolution kernel $a$.  The problem of { finding necessary and sufficient conditions on $A$, $L_i$ and $\mathscr  I\subset \{1,\dots,n\}$ for recovering  an unknown function $x$ from the dynamical samples  $Y=\{ x(i), \ldots,A^{L_i}x(i)  : i\in \mathscr I\}$  was studied in \cite{ADK13, ADK15}.} The following corollary uses the characterization obtained in \cite{ADK13}; 
it allows us to completely characterize all initial sampling sets $\mathscr I$ with $|\mathscr I|=3$ (minimal cardinality possible in this setting).

\begin{corollary} \label{posconv}
Assume $n\ge 3$ is odd,  $a \in \mathbb{R}^n$ is a  real symmetric convolution kernel such  that its discrete Fourier transform $\hat a$ is positive  and strictly decreasing on $\{0,1,\cdots,\frac{n-1}{2}\}$, and $Ax=a*x$. Assume also that $\mathscr I=\{i_1, i_2, i_3\} \subset \Z_{n}$. 
Then any signal $f \in \mathbb{R}^n$ can be uniquely determined up to a sign from the unsigned spatiotemporal samples 
\begin{equation}\label{spatiotemporal}
Y=\{ \lvert f(i) \rvert, \lvert Af(i)\rvert, \cdots, \lvert A^{n}f(i) \rvert : i\in \mathscr I\}
\end{equation} 
if and only if  $ \lvert i_k-i_j \rvert$ and $n$ are co-prime for distinct $j,k=1,2,3$.
\end{corollary}
\subsection{Infinite dimensional case}
In this section, we consider the problem of phaseless reconstruction from iterations of self-adjoint  or (more generally) diagonalizable reductive operators in a separable, infinite dimensional Hilbert space. 
%
As we mentioned in the introduction, a {\em reductive}  operator $A$ on a Hilbert space is  such that if $V$ is an invariant subspace for $A$, then $V$ is also invariant for $A^\ast$.  In particular, every self-adjoint (but not every normal) operator is reductive \cite{conway}. 

\begin{theorem}\label{GenRealMatInf}
Let $A$  be a reductive, diagonalizable operator on a real separable Hilbert space $\HH$. Assume that $A^T=B^{-1}DB$, where $B$ is a bounded operator  from $\HH$ onto $\ell^2(\N)$, and  $D$ is a diagonal operator on $\ell^2(\N)$ with the spectral decomposition $D =\sum_{j\in \Delta}\lambda_jP_j$.  Let $\Psi = \{\psi_i=B\phi_i: i\in\mathscr I\}$ be a set of vectors in $\ell^2(\N)$,  and $r_i =r^D_{\psi_i}$.  Assume that $r_{\max}=\sup \{r_i: i \in \mathscr I\}<\infty$, and let $\Lambda= (\lambda_j^\ell)_{0\le l\le 2r_{\max-2},\ j \in \Delta}$. Assume also that the set $\Psi $  is {locally complementary} with respect to $\mathscr P_D$,  and every $k\times k$ sub-matrix  of $\Lambda$ with $k\le r_{\max}$ is nonsingular. Then the set of vectors 
\begin{equation}\label{vectors3}
\mathcal E = \{\phi_i, (A^T)\phi_i,\cdots,(A^T)^{2r_i-2}\phi_i: i\in \mathscr I\}
\end{equation}
does phaseless reconstruction in $\HH$.  In other words, any $f\in\HH$ can be uniquely determined up to a sign from its unsigned measurements $$Y=\{\lvert \la A^lf, \phi_i\ra \rvert:  i \in \mathscr I,\; l=0,\dots, 2r_i-2\}.$$
\end{theorem}

As a corollary, we get the following theorem for strictly positive operators.

\begin{theorem}\label{GenRealMatInfPos}
Let $A$  be a strictly positive operator on a real separable Hilbert space $\HH$. Assume that $A^T=U^{\ast}DU$, where $U$ is a unitary operator  from $\HH$ onto $\ell^2(\N)$, and  $D$ is a diagonal operator on $\ell^2(\N)$ with the spectral decomposition $D =\sum_{j\in \Delta}\lambda_jP_j$.  Let $\Psi = \{\psi_i=U\phi_i: i\in\mathscr I\}$ be a set of vectors in $\ell^2(\N)$,  $l_i =2r^D_{\psi_i}-1$ if $r^D_{\psi_i}<\infty$ and $l_i=\infty$ otherwise.  If the set $\Psi $  is {locally complementary} with respect to $\mathscr P_D$ then the set of vectors 
\begin{equation}\label{vectors4}
\mathcal E = \{\phi_i, (A^T)\phi_i,\cdots,(A^T)^{2r_i-2}\phi_i: i\in \mathscr I\}
\end{equation}
does phaseless reconstruction in $\HH$.  In other words, any $f\in\HH$ can be uniquely determined up to a sign from its unsigned measurements $$Y=\{\lvert \la A^lf, \phi_i\ra \rvert:  i \in \mathscr I,\; 0\le l < l_i\}.$$
\end{theorem}

\section {Proofs for Section \ref {Main}}\label{mainproof}

In order to prove our main results, we need the following theorems from \cite{ACMT14}.
\begin {theorem}[\cite{ACMT14}]
\label {tadcul}
Let $J\in \CC^{n\times n}$ be a matrix in Jordan form  as in  \eqref{Jordan1}. 
Let   $\{b_i: i \in \mathscr I\}\subset \CC^n$ be a finite subset of vectors, $r_i$ be the degree of  the $(J, b_i)$-annihilator, 
  $l_i=r_i-1$, and $\mathscr P_J = \{P_s: s = 1, \dots, d\}$ be the penthouse family for $J$ introduced in Definition \ref {defJord}. 

Then the following statements are equivalent:
\begin{enumerate}

\item[(i)]
The set of vectors $\{J^jb_i: \; i \in \mathscr I, j=0,\dots, l_i\}$ is a frame for $\CC^n$.
\item[(ii)]
For every $s = 1, \dots, d$, the set $\{P_s b_i, i \in \mathscr I\}$ forms a frame for $E_s = P_s\CC^n$.\end{enumerate}
\end{theorem}

For the case of a diagonal matrix in $\ell^2(\N)$ we have 

\begin {theorem}[\cite{ACMT14}] \label {complete}
Let $D=\sum_{j\in \Delta}\lambda_jP_j$ be  a reductive diagonal operator on $\ell^2(\N)$ such that $\mathscr P_D = \{P_j:j \in \Delta\}$ is a family of projections in $\ell^2(\N)$ which forms a resolution of the identity.  Let   $\{b_i: i \in \mathscr I\}$ be a countable subset of vectors of $\ell^2(\N)$, $r_i$ be the degree of  the $(D, b_i)$-annihilator {\color {green} (see Definition \ref {minpol})}, and  $l_i=r_i-1$. 
Then the set $\big\{A^lb_i:\; i \in \mathscr I,  l=0,\dots, l_i\big\}$ is complete in $\ell^2(\N)$ if and only if for each $j\in\Delta$, the set $\big\{P_jb_i: i  \in \mathscr I\big\}$ is complete on the range $E_j$ of $P_j$. 
\end{theorem}

{ \begin {remark}
Note that, from  Definition \ref {minpol}, $r_i$ and hence $l_i$ can be infinite.
\end {remark}
}

We will also need the following lemma, which follows immediately from the definitions. 

\begin{lemma} \label {invertiblecomp} Let $B \in \mathbb {C} ^{n\times n}$ be an invertible matrix.  Then 

\begin {enumerate} 
\item The set $\mathcal{F}=\{f_1,\ldots,f_m\}$  does phaseless reconstruction in
$\mathbb K^n$ if and only if the set $B\mathcal{F}=\{Bf_1,\ldots, Bf_m\}$ does phaseless reconstruction in
$\mathbb K^n$, $\mathbb K\in\{\R, \CC\}$.

\item The set $\mathcal{F}=\{f_1,\ldots,f_m\}$  has complement property in $\mathbb K^n$ {\color{green} (see Definition \ref {CompProp})}  if and only if the set $B\mathcal{F}=\{Bf_1,\ldots, Bf_m\}$ also has the complement property in $\mathbb K^n$, $\mathbb K  \in\{\R, \CC\}$.

\item The set $\mathcal{F}=\{f_1,\ldots,f_m\}\subset \R^n$  has complement property in $\mathbb R^n$   if and only if it also has the complement property in $\mathbb C^n$.
\end {enumerate} 
\end{lemma}

\begin {proof}[Proof of Proposition \ref {NecConLC}] 
Assume that the set $\mathcal E$ in \eqref{sete} does phaseless reconstruction on $\R^n$. It follows from 
Theorem \ref{spanning} and Lemma \ref{invertiblecomp} that the set
\[
\{J^j\psi_i: i\in \mathscr I, j = 0,\ldots, 2r_i-2\}
\]
has complement property in $\CC^n$. Let us consider an arbitrary partition $\mathscr I_1,\mathscr I_2$ of $\mathscr I$. Then one of the sets $\mathscr J_k = \{J^j\psi_i: i\in \mathscr I_k, j = 0,\ldots, 2r_i-2\}$, $k=1,2$, spans $\CC^n$. Without loss of generality we may assume that the set $\mathscr J_1$ has this property.
From the definition of a $(J,\psi_i)$-annihilator, however, we have
that $\mathscr J_1 = \{J^j\psi_i: i\in \mathscr I_1, j = 0,\ldots, r_i-1\}$. Theorem \ref {tadcul} now implies that
every set $\{P_s \psi_i, i \in \mathscr I_1\}$ forms a frame for $E_s = P_s\CC^n$, $P_s\in\mathscr P_J$. Therefore, $\Psi$ is locally complementary with respect to $\mathscr P_J$.
%
%
\end {proof}

We are now ready to prove the complementarity property for the special case when $A=J$ is a Jordan matrix.

\begin{lemma}
\label {Jordan}
Let $J\in\CC^{n\times n}$ be a Jordan matrix. Let $\Psi = \{\psi_i: i\in\mathscr I\}$ be a set of vectors in $\CC^n$, and $r_i =r^J_{\psi_i}$. If  $\Psi$ is  locally complementary with respect to the penthouse partition $\mathscr P_J$, and if $J$ is  iteration regular then the set
 \[
 \mathcal F = \{J^\ell\psi_i:   i\in\mathscr I, \ell= 0,\ldots, 2r_i-2\}
 \]
 has the complement property in $\CC^n$.
\end{lemma}


\begin{proof} 
Let $\mathcal F_1,\mathcal F_2$ be a partition of the set $\mathcal F.$ 
For $\ell=1,2$, let $\Gamma_i^\ell=\{j: J^j\psi_i\in \mathcal F_\ell\}$ and
 $\I_{\ell} = \{i\in\I: |\Gamma_i^\ell| \ge r_i\}$. Clearly, $\{\I_1, \I_2\}$ forms a partition of $\I$. Therefore, due to local complementarity, we may assume that the set $\{P_j\psi_i, i\in\I_1\}$ spans $E_j$ for each $j = 1,\ldots, d$ {\color {green} (see Definition \ref {defJord})}. We will show that $\mathcal F_1$ spans $\CC^n$. 
 
 First, we claim that $$span \mathcal F_1
 =span \{J^\ell \psi_i:  i\in\I_1, \ell=0, 1\dots\} =span \{J^\ell\psi_i: i\in\I_1, \ell=0,\dots,r_i-1\}.$$ 
 The second equality follows since $r^J_{\psi_i} = r_i$. 
The first one follows from Proposition  \ref{krylovspan}, since $J$ is iteration regular and for each $i\in \I_1$, $|\Gamma_i^1|=|\{j: J^j\psi_i\in \mathcal F_1\}|\geq r_i$.
 
Finally, since the set $\{P_j\psi_i, i\in\I_1\}$ spans $E_j$ for each $j = 1,\ldots, d$, it follows from Theorem \ref {tadcul} that $ span\{J^\ell\psi_i, i\in \I_1, \ell = 0,1,\ldots\}=span \mathcal F_1=\CC^n$. 
\end{proof}

\medskip



\begin {proof} [Proof of Theorem \ref  {GenRealMatJorD}] 
To prove the theorem, we simply use Lemmas \ref{Jordan}, 
 \ref{invertiblecomp}, and Theorem \ref{spanning}.
\end {proof}

\begin{proof}[Proof of Corollary  \ref {posconv}]
Let $\{e_i :i=0,\cdots,n-1\}$ be the standard basis of $\mathbb{R}^n$ and ${F}_n= \Big(\frac {\omega^{jk}} {\sqrt{n}}\Big)$, $j,k=0,\dots,n-1$, $w=e^{\frac {-2\pi i}{n}}$, denote the $n \times n$ discrete Fourier matrix. By the convolution theorem, $A=F_n^* diag(\hat a) F_n$. By the symmetry and monotonicity  of $\hat a$, we have { $ diag(\hat a)=\sum_{k=1}^{\frac{n-1}{2}}(\hat a(k)P_k)+\hat a(0) P_0$}, where $P_k$ is the orthogonal projection from $\mathbb{C}^n$ onto $\{e_k, e_{n-k}\}$ for $k=1,\ldots, \frac{n-1}{2}$, and $P_0$ is  the orthogonal projection from $\mathbb{C}^n$ onto $\{e_0\}$. It is easy to see that the maximal dimension of  the range space of $P_k$ is 2; therefore, $\lvert \I\rvert \geq 3$ by Corollary \ref {loccompreal}. Let $b_{i}=\overline{{F_n}} e_i$.  We need to check that the conditions of  Theorem \ref{GenRealMat}, are satisfied.  Note that, since for any $i$ the vector $b_i$ is a column of $\overline {F_n}$ which has no zero entries, { $spark \{P_kb_i: i\in \I\}=2$ for $k=1,\ldots, \frac{n-1}{2}$, and $spark \{P_0b_i: i\in \I\}=1$  }. Observe that, for $k\ne n$,  $i_1,i_2$, $P_kb_{i_1}$ and $P_kb_{i_2}$ are linearly independent if and only if $\omega^{i_1-i_2}-\overline {\omega^{i_1-i_2}}=0$. This happens if and only if $(i_1-i_2)k\ne0\mod n$, $k=1,\dots, \frac {n-1} 2$, which is equivalent to the fact that $|i_1-i_2|$ and $n$ are coprime. The other set of indices is handled similarly. Thus, for all $k\ne 0$, we have $spark \{P_kb_i: i\in \I\}=3$ if and only if $ \lvert i_k-i_j \rvert$ and $n$ are co-prime for distinct $j,k=1,2,3$. Since the number of distinct eigenvalues is $\frac {n+1} 2$, and since $b_i$ has no zero entries, $r_i=\frac {n+1} 2$. Finally, using Theorem \ref {GenRealMat} {and Remark \ref {positiveEV}}, Corollary  \ref {posconv} is proved.
\end{proof}


\begin {proof}[Proof of Theorem \ref {GenRealMatInf}]
We first show that the set  $ \mathcal F = \{D^\ell\psi_i:   i\in\I, \ell \in 0,\ldots, 2r_i-2\}$
 has the complement property in $\ell^2(\N)$. 
 As in the proof of Lemma \ref {Jordan}, we let  $\mathcal F_1,\mathcal F_2$ be a partition of the set $\mathcal F,$  and
let $\Gamma_i^\ell=\{j: D^j\psi_i\in \mathcal F_\ell\}$ and
 $\I_{\ell} = \{i\in\I: |\Gamma_i^\ell| \ge r_i\}$, $\ell=1,2$. Then $\I_1, \I_2$ is a partition of $\I$. By the local complementarity  of $\Psi$ with respect to $\P_D$, we get that there exists $\ell\in\{1,2\}$ such that 
for each {$j \in \Delta$} the set
$\{P_j\psi_i : i \in \I_\ell\} $
spans the range $E_j$ of $P_j$. Without loss of generality, assume that $\ell=1$.
From this condition, it follows from Theorem \ref {complete} that the set $\{D^\ell\psi_i:   i\in\I_1, \ell \in 0,\ldots, r_i-1\}$ is complete in $\ell^2(\N)$. Since by construction, $|\Gamma^1_i|\ge r_i$ for each $i \in \I_1$, consider a set $\widetilde {\Gamma}^1_i\subset \Gamma^1_i$ such that $|\widetilde {\Gamma}^1_i|=r_i$. We claim that, $span \{D^\ell\psi_i: i \in \I_1, \; \ell \in  \widetilde {\Gamma}^1_i\}=span\{D^\ell\psi_i: i \in \I_1, \; 0\le \ell \le r_i-1\}=span\{D^\ell\psi_i: i \in \I_1, \; 0\le \ell \}$. The last equality {follows from the definition of $r_i$}. Thus, all we need to finish proving the claim is to show  that the set $\{D^\ell\psi_i: i \in \I_1, \; \ell \in  \widetilde {\Gamma}^1_i\}$ is linearly independent. To see this,  we consider the linear combination $\sum\limits_{\ell \in \widetilde {\Gamma}^1_i} c_\ell D^\ell\psi_i=q(D)\psi_i=0$, where  $q$ is a  polynomial that has degree at most $2r_i-2$. Hence, the polynomial $q$ divides the $(D,\psi_i)$-annihilator polynomial $p^D_{\psi_i}$ (see Definition \ref {minpol}  (2)).  It follows that $q$ has $r_i$ roots in common with $p^D_{\psi_i}$. Thus, $q(\lambda_{i_k})=0$ for exactly $r_i$ distinct values $\lambda_{i_k}$, $k=1,\dots,r_i$. In particular the coefficients $c_\ell$ must satisfy the system of $r_i\times r_i$ equations $\Lambda_i c=0$, where $\Lambda _i=(\lambda^\ell_{i_k})$, and where $c=(c_\ell)$. But since $\Lambda_i$ is non-singular, we must have $c=0$. Since $\widetilde {\Gamma}^1_i\subset \Gamma^1_i$, we also get that  $span \{D^\ell\psi_i: i \in \I_1, \; \ell \in  \widetilde {\Gamma}^1_i\}=span\{D^\ell\psi_i: i \in \I_1, \; 0\le \ell \le r_i-1\}=span\{D^\ell\psi_i: i \in \I_1, \; 0\le \ell \}=span \{D^\ell\psi_i: i \in \I_1, \; \ell \in   {\Gamma}^1_i\}$.

The proof of the theorem then follows  from Lemma \ref{invertiblecomp}, Theorem \ref{spanning}, and the fact that a set of real vectors has the complement property in the complex Hilbert space $\mathbb \ell^2(\N)$ if and only if it has the complement property in the real Hilbert space $\mathbb \ell^2(\N)$.
\end {proof}

The proof of Theorem \ref {GenRealMatInfPos} uses the M\"untz-Sz\'asz Theorem (see \cite{Heil11}).

\begin {theorem}[M\"untz-Sz\'asz Theorem] \label {MS}
Let $0\le n_1\le n_2\le \dots$ be an increasing sequence of nonnegative integers. Then
\begin{enumerate}
\item $\{x^{n_k}\}$ is complete in $C[0,1]$ if and only if $n_1=0$ and $\sum\limits_{k=2}^\infty 1/n_k=\infty$.
\item If  $0<a<b< \infty$, then $\{x^{n_k}\}$ is complete in $C[a,b]$ if and only if  $\sum\limits_{k=2}^\infty 1/n_k=\infty$.
\end{enumerate}
\end {theorem}

\begin {proof}[Proof of Theorem \ref {GenRealMatInfPos}]
As in the proof of Theorem \ref {GenRealMatInf},  let  $\mathcal F_1,\mathcal F_2$ be a partition of the set $\mathcal F,$  and
let $\Gamma_i^\ell=\{j: D^j\psi_i\in \mathcal F_\ell\}$ and
 $\I_{\ell} = \{i\in\I: r_i<\infty, \text {  and }|\Gamma_i^\ell| \ge r_i, \}\cup \{i\in\I: r_i=\infty, \text {  and }\sum\limits_{j\in \Gamma_i^\ell\setminus \{0\}}^\infty 1/j=\infty\}$, $\ell=1,2$.  Then, $\I=\I_1\cup\I_2$, but $\I_1,\I_2$ is not necessarily a partition of $\I$. However, since  $\I=\I_1\cup\I_2$, we can use the local complementarity  of $\Psi$ with respect to $\P_D$, to get that there exists $\ell\in\{1,2\}$ such that, 
for each { $j \in \Delta$}, the set
$\{P_j\psi_i : i \in \I_\ell\} $
spans the range $E_j$ of $P_j$. Without loss of generality, assume that $\ell=1$. 
 
Any {finite square sub-matrix of  the semi-infinite matrix $\Lambda=(\lambda_j^l)$} is non-singular, since by assumption, all the eigenvalues $\lambda_k$ are strictly positive (see \cite {JP05}). Hence, for $i\in \I_1$ with  $r_i<\infty$, using the  same argument as in  Theorem \ref {GenRealMatInf}, we get $span \{D^\ell\psi_i: i \in \I_1, r_i< \infty, \; \ell \in   {\Gamma}^1_i\}=span\{D^\ell\psi_i: i \in \I_1, r_i< \infty, \; 0\le \ell \le r_i-1\}=span\{D^\ell\psi_i: i \in \I_1, r_i< \infty, \; 0\le \ell \}$.  When $i \in \I_1$ and $r_i=\infty$, using the M\"untz-Sz\'asz one can prove   that $\overline {span} \{i\in\I_1: r_i=\infty, \text {  and }\sum\limits_{j\in \Gamma_i^\ell\setminus \{0\}}^\infty 1/j=\infty\}=\overline{span}\{D^\ell\psi_i: i \in \I_1, r_i=\infty, \; 0\le \ell \}$ as  in \cite [Proof of Theorem 3.6]{ACMT14}. Finally, the proof of the theorem then follows  from Lemma \ref{invertiblecomp}, Theorem \ref{spanning}, and the fact that a set of real vectors has the complement property in the complex Hilbert space $\mathbb \ell^2(\N)$ if and only if it has the complement property in the real Hilbert space $\mathbb \ell^2(\N)$.
\end {proof}

\section{Numerical Experiments} \label {NE}
In this section, we propose an optimization approach and report on the outcomes of  numerical experiments for the phaseless reconstruction in a diffusion-like process where  the evolution operator $A$ is a real circulant matrix. Recall that for a vector $f=(f_i) \in \mathbb{R}^n$, $\|f\|_{p}=(\sum_{i}|f_i|^p)^{\frac{1}{p}}$ for $1 \leq p<\infty$ and $\|f\|_{\infty}=\max_i |f_i|$. Given a real vector $f \in \mathcal{H}$,  we seek to recover $f$ by solving the following  nonlinear minimization problem:
\begin{equation}\label{opt1}
 \min_{ g \in \mathcal{D}_r} \sum_{t=0}^{L}\sum_{i \in \I} (|A^t g(i)|^2-|A^tf(i)|^2)^2,
 \end{equation}
where $L$ is  the sampling time instance, $\I \subset \{0,1,\cdots,n-1\}$ denotes the set of sampling locations and  $\mathcal{D}_r$ is the search region defined by   

$$\mathcal{D}_r=\{ g \in \mathcal{H}: \|g\|_{\infty}\leq r \}.$$

We used Matlab implemented function \textit{fmincon}  to solve the above optimization problem and denoted by  $f_{rec}$ the output of $\textit{fmincon}$.  We   defined the relative recovery error  by 
\begin {equation}
\label {deferr} 
\textit{Err}=\frac{\min\{\|f-f_{rec}\|_2, \|f+f_{rec}\|_{2}\}}{\|f\|_2}.
\end {equation}

It is obvious that $f$ and $-f$ are minimizers of  \eqref{opt1} and the minimal value of the objective function is 0. Since our objective function is non-convex in general,  the $\textit{fmincon}$ solver can get {trapped}  in a local minimum and this can prevent the objective function from decreasing to the {minimum} value 0.  If we know, however, that the uniqueness conditions are satisfied, once the final value of the objective function decreases to a sufficiently small value, the output $f_{rec}$ {should be} close to the target function $f$ up to a sign.  Otherwise, if the uniqueness conditions are not satisfied,  it may happen that the final value of the objective function is very close to 0, but the output $f_{rec}$ is far away from both $f$ and $-f$. In the following, we present the outcomes of numerical experiments that demonstrate the importance of uniqueness. 
 
 \subsection{Importance of uniqueness} 
We let $\mathcal{H}=\mathbb{R}^9$ and $A \in \mathbb{R}^{9\times 9}$ be {a circular convolution operator  which satisfies the conditions of Corollary  \ref {posconv}}.  The initial signal $f \in \mathbb{R}^9$ was chosen at random with every entry  uniformly distributed in $[-4,4]$. We picked one realization and fixed it as the initial signal. We chose time instances $t=\{0,\ldots,8\}$ as required by Corollary  \ref {posconv}.  Let $ \I \subset \{0,1,\cdots,8\}$ denote the initial sampling locations.  

We set sampling locations  $\I_1=\{1,2,3\}$ and $\I_2=\{1,4,7\}$. It is easy to check that  $\I_1$ satisfies the conditions proposed in Corollary  \ref {posconv}, whereas $\I_2$ does not.  In our experiments, for a fixed  initial signal $f$, we set the searching radius to $r=4$. We  ran 100 experiments and  chose a random initialization in the searching region each time. If  the final value of the objective function  is below the threshold {that we set to be $10^{-8}$}, we recorded the corresponding $\textit{Err}$ {defined in \ref {deferr}}. We plotted the results of $\I_1$ in the figure $\textbf{(a)}$ and the results of $\I_2$ in the figure $\textbf{(b)}$.

 \begin{figure}[htbp]
\label{figure1}
 \centering
    \includegraphics[width=\textwidth]{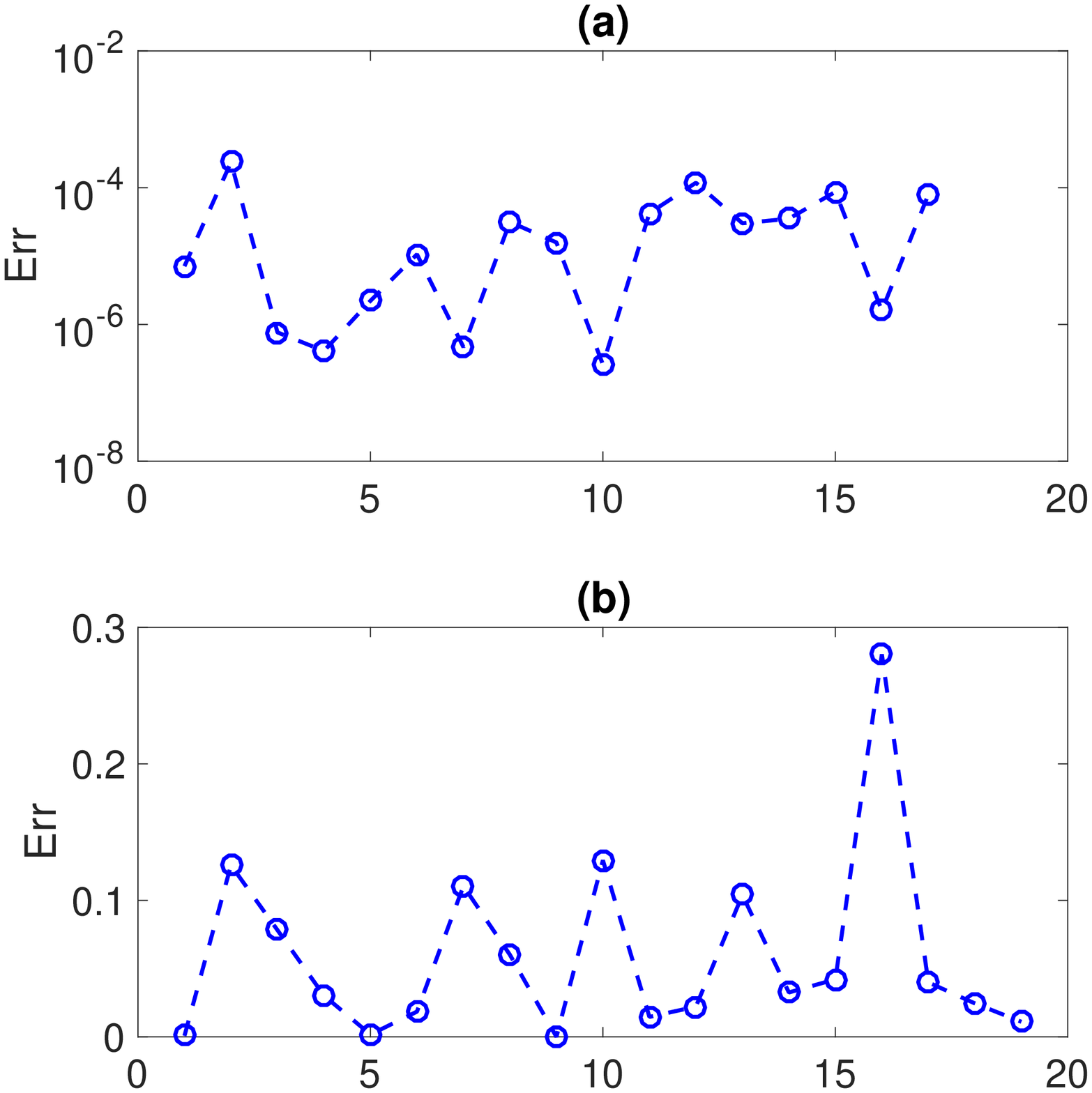}
\end{figure}

As we can see from  figure \textbf{(a)}, there are 17  out of 100 times that  final values of the objective function are below the threshold,  the relative errors were  all very small. Thus, as predicted by Corollary \ref{posconv}, we found the target function $f$ up to a sign.  In  figure $\textbf{(b)}$, there are 19  out of 100 times  that the  final values of the  objective function are below  the threshold. However, it happened in this case that most of the relative errors are large.  This is a consequence of the existence of at least one function $g\neq \pm f$ that has the same phaseless measurements. Notice also that the number of times  the algorithm converged below the threshold is larger in this case. This happened because there are more minimizers to converge to. In fact, in this case, the algorithm often converged to $\pm g$ instead to the desired functions $\pm f$. 

\subsection{A heuristic example}
In this subsection, we assume that $n$ is odd and consider an interesting case where 
\begin{equation}
\label{random}
A=F_n^*\Sigma F_n, \text{ and } \Sigma=\begin{bmatrix} \sigma_1 & 0 &\cdots&  \\ 0 & \sigma_2 &\cdots & \\ \vdots & & \ddots \\ &&& \sigma_n\end{bmatrix}
\end{equation} is a diagonal matrix whose entries are unit magnitude, complex numbers with random phases.  We generate the $\sigma_{\omega}$ as follows:
\begin{align*}
\omega=1&: \sigma_1 \sim \pm 1 \text{ with equal probability,}\\
2\leq \omega  \leq (n+1)/2&:  \sigma_{\omega}=e^{2\pi i \theta_{\omega}}, \text{ where } \theta_{\omega} \sim  \text{Uniform} [0,1]. \\
 (n+1)/2+1 \leq \omega \leq n&: \sigma_{\omega}=\sigma_{n-\omega}^*, \text{ the conjugate of } \sigma_{n-\omega}.  
\end{align*}

It is not difficult to prove the following proposition. 

\begin{proposition}\label{RandCirMat}
If $A$ is a real circulant matrix generated as in \eqref{random}, then $A$ is iteration regular with probability 1. 
\end{proposition}

For our experiments, we let $f \in \mathbb{R}^n$ be generated with every entry  drawn from the distribution Uniform[-0.5,0.5] independently and then fix it as the initial signal. We let $A$ be a realization of the random model \eqref{random}  and fix it as our evolution operator. Suppose we have noisy measurements $\{y(t,i): t=0,\cdots,2n-2, i \in \I \}$, where 
$y(t,i)=|A^tf(i)|^2+e(t,i),$
and the Gaussian noise $e(t,i) \sim \mathcal{N}(0,\sigma^2)$.   We would like to recover $f$ by solving the following minimization problem with noisy measurements

\begin{equation}\label{opt2}
 \min_{ g \in \mathcal{D}_{r}} \sum_{t=0}^{2n-2}\sum_{i \in \I} (|A^t g(i)|^2-y(t,i))^2. 
 \end{equation}
 
By Proposition $\ref{RandCirMat}$ and Theorem \ref{GenRealMatJorD}, we know that, any choice of nonempty  $\I$
guarantees the uniqueness almost surely. We will run  independent numerical experiments using the fmincon solver for different choices of $\I$.  In the noise free scenario, if the final value of the objective function \eqref{opt2} of a numerical experiment decreases to a number below a threshold, then we say that this numerical experiment is successful.  In the presence of noise, we define the threshold 
\begin{equation}\label{threshold}
v=\sum_{t=0}^{2n-2}\sum_{i \in \I} (|A^t f(i)|^2-y(t,i))^2. 
\end{equation}  If the final value of the objective function \eqref{opt2} of a numerical experiment decreases to a number below $v$, then this numerical experiment is said to be successful.  For a specific set $\I$,  we define the recovery probability $P_{\I}$ by 
$$P_{\I}=\frac{\#\text{ Successful experiments}}{\#\text{ Total experiments}}.$$

Corollary \ref{loccompreal} tells us that  the minimal number of measurements needed for real phaseless reconstruction in $\mathbb{R}^n$ is $2n-1$. We first consider the extreme case when $|\I|=1$. Then we only have $2n-1$ measurements, which is exactly the minimal requirement.  In the noise free scenario, the uniqueness conditions guarantee that $f_{rec}$ will be close to $f$ if the objective function value can decrease to a number very close to 0. Figure \ref{Figure2} displays the performance of the optimization approach for a specific example in the noise free scenario.  In this example, we observed that, for the successful numerical experiments, the objective function value decayed with iterations at a geometric rate.  The Err function decayed very slowly in the middle of iteration steps and then decayed geometrically with iterations to a number small than $10^{-1}$.  This may indicate that the objective function \eqref{opt2} is locally convex in a small neighborhood of the global minimizer.  

Next, we consider the scenarios with the presence of noise. In this case, the uniqueness is not enough.  In \cite{BW13}, it has been shown that the robust and stable phaseless reconstruction requires additional redundancy of measurements than the critical threshold, where the redundancy of measurements  is the ratio between the number of measurements  and the dimension of the signal.  In our setting, the redundancy of measurements is linearly proportional to the cardinality of $\I$.  Hence we expected that the extreme case (redundancy $\approx$ 2) would have poor robustness to noise. Figure \ref{Figure3} displays the performance of optimization approach for the example used in  Figure \ref{Figure2} with the presence of noise $(\sigma=0.01)$ and verifies our expectation.  We can see that even if the objective function value decayed  with iterations to be a number below the threshold $v$ defined in \eqref{threshold}, the Err function may not decay with the iterations and its final value is significantly large, which means that $f_{rec}$ achieved is not close to the target signal $f$.  

To obtain the numerical stability, we chose the sampling locations to be  $\I_1=\{1,2\}$ and $\I_2=\{1, 2, 3\}$ and 
$\I_i=\{1, \cdots,4i-8 \}$ for $i=3,\cdots, 6$. In Figure \ref{Figure4}, we plot  $P_{\I_i}$ and the average recovery error for $\I_i$. The result is quite striking : for a fixed problem instance, if we have sufficient  number of sampling locations ($|\I_i| \geq 3$), then the fmincon solver seems to always return a solution close to global minimizer (i.e., the target $f$ up to a sign) across many independent random initializations! This contrasts with the typical intuition of nonconvex objectives as possessing many spurious local minimizers. It would be very interesting to analyze the landscape of the objective function \eqref{opt2} similarly to the analysis in \cite{BE16, SQW16}  We leave the numerical study of this optimization approach for a future work.

\begin{figure}[htbp]
\centering
 \includegraphics[width=\textwidth]{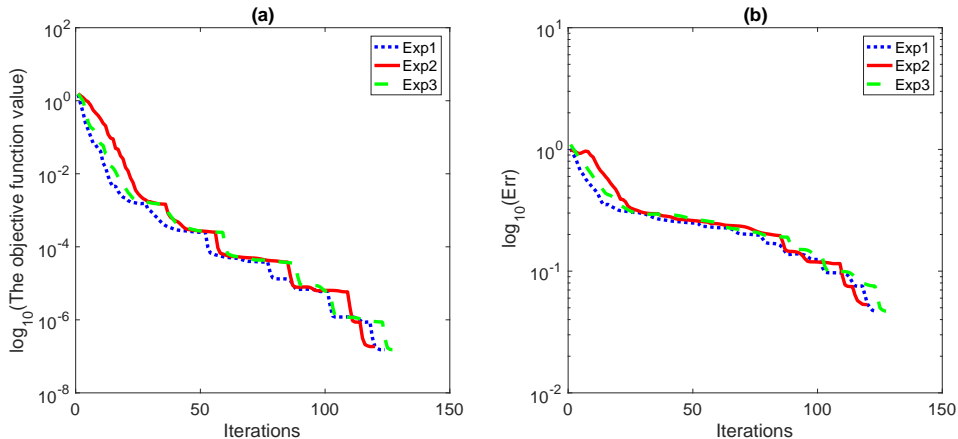}
 \caption{Let $\I=\{2\}$. We set $n=45$, $r=0.5$ and the noise level $\sigma=0$. We chose the threshold to be $10^{-6}$ and  ran independent numerical experiments  until we have 3 successful numerical experiments. We chose a random initialization in the searching region for each independent numerical experiment. We exhibit how the value of objective function \eqref{opt2} in $\log_{10}$ scale decayed with  iterations in \textbf{(a)}  and how Err in $\log_{10}$ scale decayed with  iterations in \textbf{(b)} for three successful numerical experiments.}
\label{Figure2}
 \end{figure}
 
 \begin{figure}[htbp]
\centering
 \includegraphics[width=\textwidth]{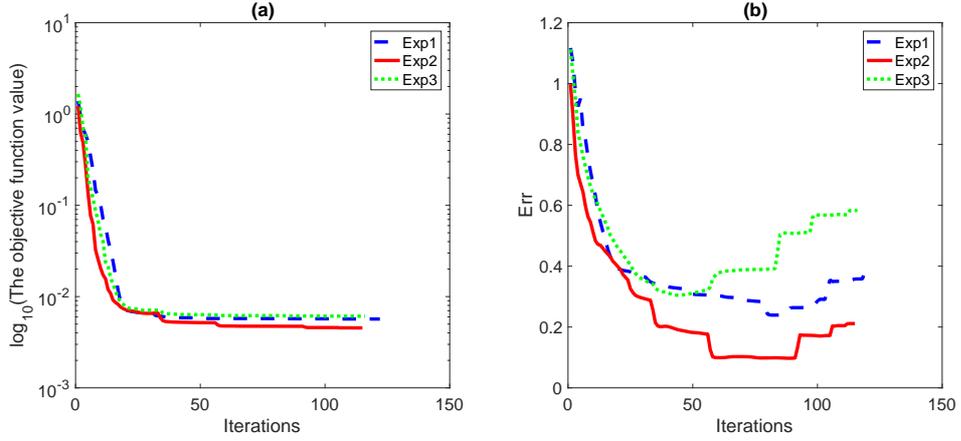}
 \caption{Let $\I=\{2\}$. We set $n=45$, $r=0.5$ and the noise level $\sigma=0.01$. We ran independent numerical experiments  until we have 3 successful numerical experiments. We exhibit how the value of objective function \eqref{opt2} in $\log_{10}$ scale decayed with  iterations in \textbf{(a)}  and how Err function in $\log_{10}$ scale behaved with  iterations in \textbf{(b)} for three successful numerical experiments. As we can see that, Err can increase even if the objective function value decreased. }
\label{Figure3}
 \end{figure}

 \begin{figure}[htb]
\centering
 \includegraphics[width=\textwidth]{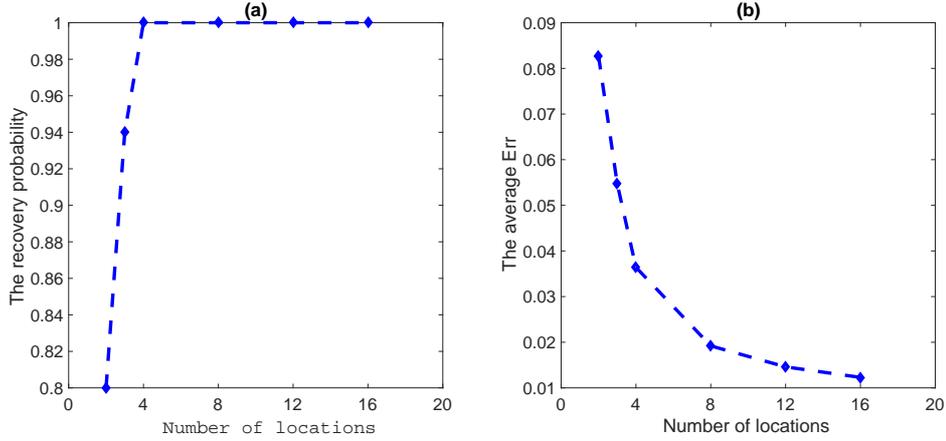}
 \caption{We set $n=45$, $r=0.5$ and the noise level $\sigma=0.01$. We  chose the sampling location $\I_1=\{1,2\}, \I_2=\{1,2,3\}$ and $\I_i=\{1,\cdots, 4i-8\}$ for $i=3,\cdots,6$. For each choice of $\I_i$, we  ran 100 independent numerical experiments and choose a random initialization in the searching region for each independent  experiment. We calculated  $P_{\I_i}$ and summarized them  in $\textbf{(a)}$ and plotted the average recovery error in $\textbf{(b)}$.  As we can see that, increasing the number of locations help increase the accuracy of the numerical solutions. Surprisingly, if we have enough number of sampling locations, any random initialization in the searching region can always converge to a solution close to a global minimizer!}
 \label{Figure4}
\end{figure}

\section{Acknowledgements}
The authors were supported in part by the collaborative NSF ATD grant DMS-1322099 and DMS-1322127. They would like to thank Rozy the cat for letting them, in a first,  do the research and write this manuscript without his supervision.

\clearpage
\bibliography{refers}{}
\bibliographystyle{siam}

\end{document}